\documentclass[12pt]{amsart}
\pdfoutput=1 
\DeclareFontFamily{OT1}{pzc}{}
\DeclareFontShape{OT1}{pzc}{m}{it}{<-> s * [1.10] pzcmi7t}{}
\DeclareMathAlphabet{\mathpzc}{OT1}{pzc}{m}{it}

\usepackage{amsrefs}
\usepackage{amsfonts}
\usepackage{amssymb}
\usepackage{amsthm}
\usepackage{amsmath}
\usepackage{tikz}
\usetikzlibrary{matrix,arrows}
\usepackage{xypic}

\DeclareMathOperator{\Hom}{Hom}

\DeclareMathOperator{\Ext}{Ext}

\DeclareMathOperator{\Tor}{Tor}
\DeclareMathOperator{\gr}{gr}
\DeclareMathOperator{\img}{im}

\DeclareMathOperator{\spn}{span}
\DeclareMathOperator{\E}{E}

\newcommand{\ktwo}{\mathcal{K}_2}

\newcommand{\T}{\mathbb{T}}

\newcommand{\bbk}{\Bbbk}

 \newsymbol\kk 207C
 \newcommand{\la}{\langle}
\newcommand{\ra}{\rangle}

\newtheorem{theorem}[subsection]{Theorem}

\newtheorem{lemma}[subsection]{Lemma}
\newtheorem{proposition}[subsection]{Proposition}

\theoremstyle{definition}
\newtheorem{definition}[subsection]{Definition}
\newtheorem{example}[subsection]{Example}

\begin{document}

\title
{Quotients of Koszul algebras and 2-d-determined algebras}

\author[Cassidy]{Thomas Cassidy}
\address{Mathematics Department, Bucknell University, Lewisburg,
PA} \email{tcasssidy@bucknell.edu}
\thanks{This work was partially supported by a grant from the Simons Foundation (Grant Number 210199 to Thomas Cassidy).
The authors would like to thank Andrew Conner for his helpful suggestions.}

\author[Phan]{Christopher Phan}
\address{Department of Mathematics and Statistics,
Winona State University, Winona, MN} \email{CPhan@winona.edu}

\setcounter{page}{1}

\begin{abstract}
Vatne \cite{Vatne} and Green \& Marcos \cite{greenmarcos} have
independently studied the Koszul-like homological properties of
graded algebras that have defining relations in degree 2 and
exactly one other degree. We contrast these two approaches, answer
two questions posed by Green \& Marcos, and find  conditions that
imply the corresponding Yoneda algebras are generated in the
lowest possible degrees.
\end{abstract}

\vskip-1in \maketitle

 \vspace{0.2in}

\section{Introduction}

The Koszul property for graded associative algebras \cite{priddy}
has been generalized in several directions,  including algebras
with defining relations in just one degree ($d$-Koszul algebras
\cite{berger}) and algebras with defining relations in multiple
degrees ($\ktwo$ algebras \cite{csk2alg}).  $d$-Koszul algebras
share many of the nice homological properties of Koszul algebras,
but are not closed under several standard operations. The family
of $\ktwo$ algebras, which includes the $d$-Koszul algebras, has
the advantage of being closed under graded Ore extensions, regular
normal extensions, and tensor products.  Interpolating between
$d$-Koszul and $\ktwo$ algebras leads one to look for Koszul-like
homological properties in algebras with defining relations in just
two degrees.  This idea was considered independently by   Vatne
\cite{Vatne} and Green \& Marcos \cite{greenmarcos}, who each
investigate graded algebras with defining relations in degree 2
and exactly one other degree.  We compare these two approaches to
find sufficient conditions for such an algebra to be $\ktwo$, and
answer two questions posed by Green \& Marcos.

Let $\kk$ be a field and $d$ an integer greater than 2. We
consider graded $\kk$-algebras of the form $R=\T(V)/I$ where $V$
is a finite dimensional vector space, $\T(V)$ is the free algebra
generated by $V$ and  $I$ is a homogenous ideal which can be
generated  by elements in $\T(V)_2$ and $\T(V)_d$. Properties of
$I$ determine the structure
  the bi-graded Yoneda algebra   $E(R):=
\oplus_{i,j} \Ext_R^{i,j}(\kk,\kk)$, where $i$ refers to the cohomological
degree and $j$ is the internal degree that $E(R)$ inherits from the grading on $R$.
 We let $E^i(R)$ denote $\oplus_j E^{i,j}(R)$.
\begin{definition}\label{d:koszul}  $R$ is Koszul if  $E(R)$ is generated as a $\kk$-algebra by $E^{1}(R)$.
  $R$ is $\ktwo$ if  $E(R)$ is generated  as a $\kk$-algebra by $E^1(R)$ and $E^2(R)$. $R$ is $d$-Koszul if
  $E(R)$ is generated  as a $\kk$-algebra by $E^1(R)$ and $E^2(R)$, and also  $I$ is generated in  degree $d$.
\end{definition}

Each of these definitions requires $E(R)$ to be generated in the
lowest possible degrees.  One  can determine whether or not an
algebra is Koszul or $d$-Koszul just by knowing the bi-degrees in
the corresponding Yoneda algebra.
 (Specifically, for $A$ to be $d$-Koszul, we need $E^{i,j}(A) = 0$ unless $j = \delta(i)$, where $\delta(2m) = dm$ and
$\delta(2m +1) = dm+1$. Note that Koszul and 2-Koszul are
synonomous.) In contrast, $E(R)$ can have the same bi-degrees as
an algebra generated in degrees 1 and 2  even when $R$ is not
actually $\ktwo$.

To explore the connections between these definitions, it is
helpful to consider the quadratic generators of $I$ separately
from the degree $d$ generators. Let $I_2$ denote a linearly
independent set of quadratic relations, and  $J$ a set of degree
$d$ relations, so that $I$ is the ideal generated by $I_2$ and
$J$.
 Note that different choices for $J$ can produce the same algebra $R$.
Let $A$ be the algebra $\T(V)/\la I_2\ra$ and let $B = \T(V)/\la
J\ra$, so that the algebra $R$ can be realized as either   $A/\la
J\ra$ or $B/\la I_2\ra$.

 \section{Almost linear resolutions}

 It would be nice if the Koszul property of $A$ and the $d$-Koszul property of $B$ would
 combine to imply that $R$ is $\ktwo$, but this is not necessarily the case,
as is shown in example \ref{counter} below. Indeed, any two of the
algebras $A$, $B$ and $R$ can have good homological behavior while
the third is recalcitrant.
 Remark 7.5 in \cite{ConShel} illustrates the case where $R$ is $\ktwo$, $B$ is $3$-Koszul, and yet $A$ is not Koszul.  In the following example, $R$ is $\ktwo$, $A$ is Koszul, but the Yoneda algebra of $B$ fails to be generated in low degrees.
 \begin{example}
 Let $V=\{x,y\},$ $I_2=\{xy-yx\}$ and $J=\{xyx\}$.  Then $A$ is commutative and $xyx$ is a regular element in $A$, hence by \cite{csk2alg}*{Corollary 9.2}, $R$ is $\ktwo$. But $B$, as a monomial algebra,  fails to be  $3$-Koszul by Proposition 3.8 in \cite{berger}.
 \end{example}

Clearly different hypotheses are required to get good behavior from $R$.
In \cite{Vatne}, Vatne considers the case were $A$ is Koszul, and $R$ has an {\it almost linear} resolution as an $A$-module.
\begin{definition}  $R$ has an almost linear resolution as an $A$-module if $\Ext_A^{i}(R,\kk)=\Ext_A^{i,d-1+i}(R,\kk)$ for all $i>0$.
\end{definition}

Vatne shows that if $A$ is Koszul, $d>3$ and $R$ has an almost
linear resolution as an $A$-module, then $E(R)$ has the correct
bi-degrees for $R$ to be a $\ktwo$ algebra. In fact $R$ is $\ktwo$
in this case, as the following  direct corollary of Theorem 5.15
in \cite{ConShel} shows.
\begin{proposition}\label{vatne-answer}
If $A$ is Koszul, $d \geq 3$, and $R$ has an almost linear resolution as an $A$-module, then $R$ is $\ktwo$.
\end{proposition}

In this case of monomial algebras, the almost linear condition is
relatively easy to check. This fact motivates our concluding
theorem.

\begin{proposition}\label{monom}  Suppose that $I_2$ and $J$ consist of monomials  and that no element of $J$ contains any element
 of $I_2$ as a connected  subword. Then $_AR$ has an almost linear resolution if and only if $\Ext_A^1(R,\kk)=\Ext^{1,d}_A(R,\kk)$.  In this case, $R$ is $\ktwo$.
\end{proposition}

\begin{proof}
If $_AR$ has an almost linear resolution, then  by definition\\
$\Ext^1_A(R, \bbk) = \Ext^{1,d}_A(R, \bbk)$.

Now, suppose $\Ext^1_A(R, \bbk) = \Ext^{1,d}_A(R, \bbk)$. In the
remainder of this proof, let $\pi_A : \T(V) \rightarrow A$ be
defined by $\pi_A(s) := s + \la I_2 \ra$. Let $\mathcal{M}$ be the
set of monomials $u \in \T(V)$ with $\pi_A(u) \not = 0$.

Note that $A \otimes A_+^{\otimes \bullet}$ is a projective
resolution of $_AR$. We will construct a subresolution $A \otimes
Q_\bullet$ of $A \otimes A_+^{\bullet}$, which is a minimal
projective resolution. We construct $Q_\bullet$ by choosing a
monomial basis $\mathcal{B}_i$ for each vector space $Q_i$.
Because $\Ext^1_A(R, \bbk) = \Ext^{1,d}_A(R, \bbk)$, the
left-ideal in $A$ generated by   $\pi_A(J)$ is equal to the
two-sided ideal generated by the same elements. Thus we may begin
by setting $\mathcal{B}_1 = \{\pi_A(s): s \in J\}$.

Now, suppose that $\mathcal{B}_i$ consists of elements of the form
\[\pi_A(u_i) \otimes \pi_A(u_{i-1}) \otimes \cdots \otimes \pi_A(u_1)\]
where each $u_t \in \mathcal{M}$. Then, we   set
\[\begin{split}\mathcal{B}_{i+1} := \{ &\pi_A(u_{i+1}) \otimes \pi_A(u_i) \otimes \pi_A(u_{i-1}) \otimes \cdots \otimes \pi_A(u_1) \\
&: u_{t} \in \mathcal{M}, \pi_A(u_i) \otimes \pi_A(u_{i-1}) \otimes \cdots \otimes \pi_A(u_1) \in \mathcal{B}_i,\\
 &\text{ and } \pi_A(u_{i+1}) \text{ is a minimal left-annihilator of $\pi_A(u_i)$.}\}\end{split}\]
Therefore, every $\mathcal{B}_i$ will consist of pure tensors of
monomials. Furthermore, since $A$ is a quadratic monomial algebra,
each minimal left-annihilator is linear, and so
\[\mathcal{B}_i \subset A_1^{\otimes i -1} \otimes \spn_\bbk J.\]
Therefore, $A \otimes Q_\bullet$ is an almost linear resolution of
$_AR$. Since $A$ is monomial, it is also Koszul, and thus $R$ is
$\ktwo$ by \ref{vatne-answer}.
\end{proof}

\section{2-$d$-determined algebras}

In \cite{greenmarcos}, Green and Marcos study the case where the bi-degrees of $E(R)$ are no greater than
the bi-degrees of a $\ktwo$ algebra with defining relations in degree 2 and $d$.  They call such an algebra 2-$d$-determined.
 Like Vatne,  Green and Marcos
 assume that $R$ is a quotient of a Koszul algebra $A$, but they do so via Gr\"obner bases.
They assume that $I$ has a reduced Gr\"obner basis $\mathpzc{g}=\mathpzc{g}_2\cup \mathpzc{g}_d,$ so that
 $A=\T(V)/\la \mathpzc{g}_2\ra$ and $B=\T(V)/\la \mathpzc{g}_d\ra$.

At the end of \cite{greenmarcos}, Green and Marcos ask three questions.  In Theorems \ref{answer1} and \ref{answer2}
we provide negative answers to the first two questions (which we have rephrased slightly)\footnote{In their formulation,
Green and Marcos consider quotients of graph algebras. We only consider connected-graded algebras, which suffice to answer
the questions in the negative.}:
\begin{enumerate}
\item If $C$ is an $2$-$d$-determined, then is the $\Ext$-algebra $\E(C)$ finitely generated? \label{q:fingen}
\item If $C$ is a $2$-$d$-determined algebra and the $\Ext$-algebra $\E(C)$ is finitely generated, is $C$ a $\ktwo$
algebra (assuming that the global dimension of $C$ is infinite)?\label{q:k2}
\end{enumerate}

We use the following construction (see \cite{ppquadalg}*{\S III.1}):
\begin{definition}
Let $A$ and $B$ be graded algebras. The \textbf{free product} of $A$ and $B$ is the algebra
\[A \sqcup B := \bigoplus_{\substack{i \geq 0\\ \epsilon_1, \epsilon_2 \in \{0, 1\}}} A_+^{\otimes \epsilon_1} \otimes (B_+ \otimes A_+)^{\otimes i} \otimes B_+^{\otimes \epsilon_2}.\]
\end{definition}
This related result will be of importance.
\begin{proposition}[c.f. {\cite{ppquadalg}*{Proposition III.1.1}}]\label{p:freeprodext}
For graded algebras $A$ and $B$,
\[\E(A\sqcup B) \simeq \E(A) \sqcap \E(B).\]
\end{proposition}


Suppose  $V$ has an ordered basis $x_1 < x_2 < \cdots < x_n$. Then we can order the monomials of $\T(V)$ by
degree-lexicographical order. This induces a filtration $F$ on $A$, $\Tor_A(\bbk, \bbk)$ and $\E(A)$.
(See, for example, \cite{ppquadalg}*{Chapter IV} and \cite{phan}.)

With this filtration $F$, we now have several versions of the
$\Ext$-functor. First, there is the \emph{ungraded} $\Ext$-functor
(over $A$-modules and $\gr^F A$-modules), which is the derived
functor of the ungraded $\Hom$-functor. Next, since $A$ and its
associated graded algebra $\gr^F A$ are both graded with respect
to an internal degree, we have the $\mathbb{N}$-\emph{graded}
$\Ext$-functor (over $A$-modules and $\gr^F A$-modules), the
derived functor of the $\mathbb{N}$-graded $\Hom$ functor. Finally
$\gr^F A$ is graded by the monoid of monomials in $\T(V)$, and
hence we have a monomial-graded $\Ext$-functor (over $\gr^F
A$-modules).

We will use the following result:
\begin{lemma}[{\cite{phan}*{Theorem 1.2}}]\label{l:bigradedmono}
Let $\gr^F \E(A)$ be the associated graded algebra of \emph{ungraded} algebra $\E(A)$ under the filtration $F$, and
$\E(\gr^F A)$ be the \emph{graded} $\Ext$-algebra with respect to the monomial grading. Then there is a bigraded
 algebra monomorphism
\[\Lambda: \gr^F \E(A) \hookrightarrow \E(\gr^F A).\]
\end{lemma}

In the case where each $\E^{i,j}(A)$ is finite-dimensional, the graded and ungraded versions coincide
(see \cite{phan}*{Lemma 1.4}). A slight modification of the proof of \cite{phan}*{Lemma 2.11} yields:
\begin{lemma}\label{l:gradedtrans}
If $\E^{i,j}(A)$ is finite-dimensional for every $i,j$, then
\[\dim \E^{i,j}(A) = \dim \bigoplus_{|\alpha| = j} (\gr^F \E^i(A))_\alpha\] where $|\alpha|$ is the length of the monomial
$\alpha$.
\end{lemma}

To answer to question (\ref{q:fingen}),
let
\begin{align*}
A &:= \frac{\bbk\left<a, b, c, d, e, f, l, m\right>}{\left<bc - ef, ae, da-lm, cl\right>}, \\
B &:= \frac{\bbk\left<z\right>}{ \left<z^4\right>},\\
\intertext{and}
C &:= A \sqcup B.
\end{align*}
We order the monomials in
\[\bbk\left<a, b, c, d, e, f, l, m, z\right>\]
by degree-lexicographical order. This creates a filtration $F$ on $A, B,$ and $C$, as well as their corresponding $\Ext$-algebras. Using Bergman's diamond lemma in \cite{bergman}, we see that
\[\gr^F A \simeq \frac{\bbk\left<a, b, c, d, e, f, l, m\right>}{\left<ef, ae, lm, cl,  abc, cda, \right>}.\]

Consider the structure of $\E(\gr A)$. We construct a basis for a
vector space $V_\bullet \subseteq (\gr A)_+^\bullet$ so that $\gr
A \otimes V_\bullet \subseteq \gr A \otimes (\gr A)_+^\bullet$ is
a minimal projective resolution of $_{\gr A}\bbk$. This is done by
applying the left annihilator algorithm described   in
\cite{csk2alg}.   Figure \ref{f:csgraphfingen} depicts the minimal
left annihilation of  monomials in the basis for $V_\bullet$.
 Consider all paths
ending with a first-degree monomial. By tensoring the vertices in
all such paths
 we obtain a basis for
$V_\bullet$. For example, $ab \otimes cd \otimes ab \otimes c$ is
a basis element for $V_4$. Other examples and applications of
these graphs can be found in \cite{wakeforest}, where the graphs
are used to characterize finiteness properties of Yoneda algebras.
 In Figure \ref{f:csgraphfingen}, the dotted lines represent
left-annihilation from an inessential relation (e.g. $ab$
left-annihilates $cd$ because of the essential relation $abc$).

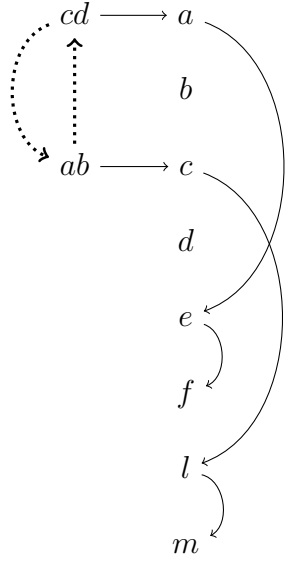
\begin{figure}
\begin{tikzpicture}
\matrix (m) [matrix of math nodes, row sep=1em,
column sep=2em, text height=1.5ex, text depth=0.25ex]
{cd & a\\
& b \\
ab & c \\
& d \\
& e \\
& f \\
& l \\
& m \\
 }; \path[->] (m-1-1) edge (m-1-2); \path[->] (m-1-1) edge
[bend right=70, dotted, very thick] (m-3-1);
\path[->] (m-3-1) edge [dotted, very thick] (m-1-1);
\path[->] (m-3-1) edge (m-3-2);
\path[->] (m-1-2) edge [bend left=70] (m-5-2);
\path[->] (m-3-2) edge [bend left=70] (m-7-2);
\path[->] (m-5-2) edge [bend left=70] (m-6-2); \path[->] (m-7-2)
edge [bend left=70] (m-8-2);
\end{tikzpicture}
\caption{A basis for $V_\bullet$ can be constructed from the paths ending in a first-degree monomial in this graph.\label{f:csgraphfingen}}
\end{figure}

Likewise, we set $W_{2i} := \bbk (z^3 \otimes z)^{\otimes i}$ and $W_{2i+1} := \bbk z \otimes W_{2i}$, so
that $B \otimes W_\bullet \subseteq B \otimes B_+^{\otimes \bullet}$ is a minimal projective resolution of
$_B\bbk$. (As with $V_\bullet$, we can visualize a basis for $W_\bullet$ using Figure \ref{f:csgraphfingenB}.)

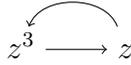
\begin{figure}[h]
\begin{tikzpicture}
\matrix (m) [matrix of math nodes, row sep=1em,
column sep=2em, text height=1.5ex, text depth=0.25ex]
{z^3 & z\\};
\path[->] (m-1-1) edge (m-1-2);
\path[->] (m-1-2) edge [bend right=70] (m-1-1);
\end{tikzpicture}
\caption{A basis for $W_\bullet$ can be constructed from the paths ending in the first-degree monomial $z$ in this
graph.\label{f:csgraphfingenB}}
\end{figure}

\begin{theorem}\label{answer1} $C$ is $2$-$4$-determined, but $E(C)$ is not finitely generated.\end{theorem}
\begin{proof}
First, it's clear that the ideal of relations for $C$ is generated
in degrees $2$ and $4$.

Now, by examining Figure \ref{f:csgraphfingen}, we see that
$\E^{i,j}(\gr^F A) = 0$ unless  $j \leq 2i - 1.$ Thus, by Lemmas
\ref{l:bigradedmono} and \ref{l:gradedtrans}, we know that
$\E^{i,j}(A) = 0$ unless $j \leq 2i-1$. Likewise, we see that
$\E^{i,j}(B) = 0$ unless $j \leq \delta(i)$ (where $\delta$ is as
defined in the paragraph after Definition \ref{d:koszul}). Now,
since $2i-1 \leq \delta(i)$ for every $i$, applying Proposition
\ref{p:freeprodext}, we see that $C$ is $2$-$4$-determined.

By Proposition \ref{p:freeprodext}, to show that $\E(C)$ is not
finitely generated, it suffices to show that $\E(A)$ is not
finitely generated. We shall exhibit a projective resolution for $_A \bbk$. Consider the sequence
\begin{equation}\label{e:resfora}
\begin{split}\cdots \xrightarrow{M_{i+1}} A(-2i+2)^{2} \xrightarrow{M_i} \cdots
\xrightarrow{M_5} A(-6)^2 \xrightarrow{M_4} A(-4)^2 \xrightarrow{M_3} A(-2)^{4}\\
 \xrightarrow{M_2} A(-1)^{8} \xrightarrow{M_1} A \xrightarrow{M_0} \bbk \rightarrow 0\end{split}
\end{equation}
where the maps are  right multiplication by the matrices
 \setcounter{MaxMatrixCols}{20} 
\begin{align*} M_1 &= (a, b, c, d, e, f, l, m)^T,\\
M_2 &= \left(\begin{matrix}
0 & 0 & 0 & 0 & a & 0 & 0 & 0  \\
0 & 0 & 0 & 0 & 0 & 0 & c & 0  \\
0 & 0 & b & 0 & 0 & -e & 0 & 0  \\
-d & 0 & 0 & 0 & 0 & 0 & 0 & l
\end{matrix}\right),\\
M_3 &= \left(\begin{matrix}
cd & 0 & 0 & 0   \\
0 & ab & 0 & 0  \\
\end{matrix}\right),\\
M_{2i} & = \left(\begin{matrix}
ab & 0 \\
0 & cd
\end{matrix}\right) \text{ for $i \geq 2$, and}\\
M_{2i+1} & = \left(\begin{matrix}
cd & 0 \\
0 & ab
\end{matrix}\right) \text{ for $i \geq 2$.}
\end{align*}

Given any fixed integer $N$, noncommutative Gr\"obner bases can be used to calculate the dimensions of $\E^{i,j}(A)$ for $j < N$ . We have used the computer program {\tt Bergman} \cite{Bergmansoftware} for this purpose.

Note that (\ref{e:resfora}) satisfies $\img M_i \subset \ker M_{i-1}$. Indeed, (\ref{e:resfora}) is clearly exact up to the matrix $M_2$. Gr\"obner bases calculations show that $\E^{3,3}(A) = 0$ and $\E^{3,5}(A) = 0$. From above, we know that $\E^{3,j}(A) = 0$ for $j > 5$, and so (\ref{e:resfora}) is exact up to the matrix $M_3$.

Likewise, Gr\"obner bases calculations show that $\E^{4, 4}(A)$, $\E^{4,5}(A) = 0$, and $E^{4,7}(A) = 0$. From above, we know that $\E^{4,j}(A) = 0$ for $j > 7$, and so (\ref{e:resfora}) is exact up to the matrix $M_4$.

However, $\img M_4 = \ker M_3$ implies that $\img M_{i+1} = \ker
M_i$ for all $i \geq 3$. It follows  that (\ref{e:resfora}) is
exact.

Therefore, we have
\[\E^{i}(A) = \begin{cases}
\E^{i, i}(A) & \text{if $i = 0, 1$,}\\
\E^{i, 2i - 2}(A) & \text{if $i \geq 2$.}
\end{cases}
\]

Fix $i \geq 3$. If the
elements of $\E^{i, 2i-2}(A)$ were generated by lower
cohomological-degree elements of $\E(A)$, then the multiplication
map
\begin{equation} \label{e:multmap} \bigoplus_{\substack{0 < j < i\\ 0< k < 2i-2}} \E^{j, k}(A) \otimes \E^{i - j, 2i - k - 2}(A) \rightarrow \E^{i, 2i-2}(A)\end{equation}
would be surjective.  We will show that this is not the case.

Suppose $2 \leq j < i$. Then for $E^{j,k}(A)$ to be nonzero, we
would need $k = 2j - 2$. This implies $2i - k - 2 = 2(i-j)$, which
is neither $i-j$ nor $2(i - j) - 2$, and so $E^{i-j,2i-k-2}(A)=0$.

On the other hand, suppose $j = 1$. Then for $E^{1,k}(A)$ to be
nonzero, we would need $k = 1$.  This implies  $2i - k - 2 = 2i -
3$. However, $i - j = i - 1 \geq 2$, and so
$E^{i-j,2i-k-2}(A)=E^{i-1,2i-3}(A)=0$.


In either case, the map in (\ref{e:multmap}) is zero for $i \geq
3$. Hence, for $i \geq 3$, $\E^i(A)$ is not generated by lower
cohomological-degree elements of $\E(A)$. Therefore, $\E(A)$ is
not finitely generated.
\end{proof}

Now we answer question (2).
\begin{theorem}\label{answer2} There exists a 2-$d$-determined algebra (with infinite global dimension) which is not
$\ktwo$ even though its Yoneda algebra is finitely generated.
\end{theorem}
\begin{proof} Let
\begin{align*}
A &:= \frac{\bbk\left<a, b,n, p, q, r, s, t, u, v, w, x, y\right>}{\left<\begin{matrix}np-nq, np-nr, ps-pt, qt-qu, rs-ru, sv- sw, tw-tx,\\ uv-ux, sv-sy, tw-ty, ux-uy, va-vb, wa-wb, xa-xb\end{matrix}\right>}, \\
B &:= \frac{\bbk\left<z\right>}{ \left<z^4\right>},\\
\intertext{and}
C &:= A \sqcup B.
\end{align*}

The algebra $A$ appears in \cite{cassidy}, and the following minimal projective resolution for $_A\bbk$ is given:
\[0 \rightarrow A(-5) \rightarrow A(-3)^7 \rightarrow A(-2)^{14} \rightarrow A(-1)^{13} \rightarrow
A.\] Note that $A$ is a quadratic algebra which is not Koszul, and
so $A$ is not $\ktwo$. However, as $A$ has finite global
dimension, $E(A)$ must be finitely generated. As shown above, the
algebra $B$ satisfies $\E^{2i}(B) = \E^{2i, 4i}(B)$, $\E^{2i+1}(B)
= \E^{2i+1, 4i+1}(B)$, and $\dim \E^i(B) = 1$.  $B$ has infinite
global dimension, but since it is $\ktwo$, $E(B)$ is
finitely-generated. It follows from   Proposition
\ref{p:freeprodext},  that $C$ is $2$-$4$-determined, has infinite
global dimension, and is not $\ktwo$.
\end{proof}

There is a small omission in the statement of Theorem 20 in
\cite{greenmarcos}. From their earlier proofs it is clear the
authors intended to include the hypotheses  ``\dots    and
$\mathpzc{g}_d$ is a Gr\"obner basis for the ideal it generates.''
The following example shows that without this hypothesis, the
conclusions of Theorem 20 are not valid.

\begin{example}\label{counter} Let $V=\{a,x,y,z\}$, ordered by $z>y>x>a$.    Then
$\mathpzc{g}=\mathpzc{g}_2\cup \mathpzc{g}_4=\{xa,az,ay\}
\cup\{y^2z^2, x^2y^2+a^4\}$ forms a   Gr\"obner basis.
$A=\displaystyle\frac{\kk\la a,x,y,z\ra}{\la \mathpzc g_2\ra}$ is
a quadratic monomial algebra and hence Koszul, and
$B=\displaystyle\frac{\kk\la a,x,y,z\ra}{\la \mathpzc g_d\ra}$ has
global dimension 2 and thus is 4-Koszul.  But over the algebra
$$R= \frac{\kk\la a,x,y,z\ra}{\la
x^2y^2+a^4,y^2z^2,xa,az,ay\ra},$$
  the module $\kk$ has a minimal projective resolution of the form
\[0 \rightarrow R(-6) \rightarrow R(-7)\oplus R(-6)\oplus R(-5)\oplus R(-3)^2 \rightarrow R(-4)^2\oplus R(-2)^{3}\\ \rightarrow \] \[ R(-1)^{4} \rightarrow R \rightarrow \kk \rightarrow 0.\]
This means that $E(R)^{3,6}\ne 0$ and so $R$  is neither  $\ktwo$
nor 2-4-determined.
 \end{example}

We can establish sufficient conditions for $R$ to be $\ktwo$ by
merging hypotheses from Vatne with those of Green and Marcos. Our
last theorem shows that we need only slightly stronger hypotheses
on the Gr\"obner basis to guarantee that $R$ is $\ktwo$.
\begin{proposition} Let  $\mathpzc{g}_2$, $\mathpzc{g}_d$ and  $\mathpzc{g}_2\cup \mathpzc{g}_d$  be
 Gr\"obner bases such that
 $\mathpzc{g}_2\cup \mathpzc{g}_d$    has no redundant elements
 (i.e.   any set of
 defining relations for $R$ has at least $|\mathpzc g_2|+|\mathpzc g_d|$
 elements).  If either
\begin{enumerate}
\item[(i)] $B=\T(V)/\la \mathpzc g_d \ra$ is $d$-Koszul, or\\
\item[(ii)] $\Ext_{\gr A}^1(\gr R,\kk)=\Ext^{1,d}_{\gr A}(\gr R,\kk)$,\\
\end{enumerate}
 then $R$ is $\ktwo$.
\end{proposition}
\begin{proof}  Using the filtration defined by the Gr\"obner bases we create the monomial algebras $\gr A$, $\gr B$ and
$\gr R$. We will show that $\gr R$ is a $\ktwo$ algebra.

 For case $(i)$, assume that $B$
is $d$-Koszul.  Then   $grB$  is $d$-Koszul by
\cite{greenmarcos}*{Theorem 10}.
 It follows from \cite{greenmarcos}*{Theorem 14} that $\gr R$ is 2-$d$-determined, and hence $\ktwo$
 by \cite{greenmarcos}*{Theorem 16}.

For case $(ii)$, if
$\Ext_{\gr A}^1(\gr R,\kk)=\Ext^{1,d}_{\gr A}(\gr R,\kk)$ then   by
proposition \ref{monom}, $\gr R$ has an almost linear resolution
over $\gr A$ and is therefore $\ktwo$.

 In either case, $\gr R$ is a monomial $\ktwo$ algebra. Since $\mathpzc g_2 \cup\mathpzc g_d$ has
  no redundant elements, it forms an essential Gr\"obner basis in the sense of \cite{phan}.
  Thus by \cite{phan}*{Theorem 1.7},  $R$ itself is $\ktwo$.
\end{proof}

 Condition $(ii)$ is certainly not necessary. Is
condition $(i)$ necessary?

\bibliographystyle{alpha}
\bibliography{CassidyPhan}

\end{document}